\documentclass[11pt]{amsart}

\usepackage{amssymb}
\usepackage{amsmath}
\usepackage{amsthm}
\usepackage{enumerate}
\usepackage{graphicx}
\usepackage{blkarray}
\usepackage{mathrsfs}

\usepackage{caption}
\usepackage{subcaption}

\usepackage[hidelinks]{hyperref}

\theoremstyle{plain}
\newtheorem{thm}{Theorem}[section]
\newtheorem{prop}[thm]{Proposition}
\newtheorem{lem}[thm]{Lemma}
\newtheorem{cor}[thm]{Corollary}

\numberwithin{equation}{section}

\theoremstyle{definition}

\theoremstyle{remark}

\newtheorem*{acknowledgements}{Acknowledgements}

\theoremstyle{plain}

\newcommand{\thmref}[1]{Theorem~\ref{#1}}
\newcommand{\propref}[1]{Proposition~\ref{#1}}
\newcommand{\secref}[1]{Section~\ref{#1}}

\newcommand{\lemref}[1]{Lemma~\ref{#1}}
\newcommand{\corref}[1]{Corollary~\ref{#1}}
\newcommand{\figref}[1]{Figure~\ref{#1}}

\newcommand{\eqnref}[1]{Equation~\eqref{#1}}

\newcommand{\calC}{{\mathcal C}}

\newcommand{\calH}{{\mathcal H}}

\newcommand{\T}{{\mathcal T}}

\newcommand{\calV}{{\mathcal V}}

\newcommand{\NN}{{\mathbb N}}
\newcommand{\CC}{{\mathbb C}}

\newcommand{\RR}{{\mathbb R}}
\newcommand{\PP}{{\mathbb P}}
\newcommand{\ZZ}{{\mathbb Z}}

\newcommand{\st}{{\, \big| \,}}

\newcommand{\param}{{\mathchoice{\mkern1mu\mbox{\raise2.2pt\hbox{$
\centerdot$}}
\mkern1mu}{\mkern1mu\mbox{\raise2.2pt\hbox{$\centerdot$}}\mkern1mu}{
\mkern1.5mu\centerdot\mkern1.5mu}{\mkern1.5mu\centerdot\mkern1.5mu}}}

\DeclareMathOperator{\ext}{EL}
\DeclareMathOperator{\area}{area}
\DeclareMathOperator{\im}{Im}

\DeclareMathOperator{\mcg}{MCG}
\newcommand{\mf}{\mathcal{MF}}

\newcommand{\scc}{\calC}
\newcommand{\gm}{\overline{\T}^\mathrm{\scriptscriptstyle GM}}
\newcommand{\teich}{\T}
\DeclareMathOperator{\SL}{SL}

\newcommand{\GM}{{Gardiner--Masur}}

\begin{document}

\title[A divergent horocycle in the horofunction compactification]{A divergent horocycle in the horofunction compactification of the Teichm\"uller metric}
\author{Maxime Fortier Bourque}
\address{School of Mathematics and Statistics, University of Glasgow, University Place, Glasgow, United Kingdom, G12 8QQ}
\email{maxime.fortier-bourque@glasgow.ac.uk}

\begin{abstract}
We give an example of a horocycle in the Teichm\"uller space of the five-times-punctured sphere that does not converge in the Gar\-di\-ner--Masur compactification, or equivalently in the horofunction compactification of the Teichm\"uller metric. As an intermediate step, we exhibit a simple closed curve whose extremal length is periodic but not constant along the horocycle. The example lifts to any Teichm\"uller space of complex dimension greater than one via covering constructions.
\end{abstract}

\maketitle

\section{Introduction}

In \cite{GardinerMasur}, Gardiner and Masur defined a compactification of Teichm\"uller space which mi\-mics Thurston's compactification \cite{Thurston}, but uses extremal length instead of hyperbolic length. Since the Teichm\"uller distance can be computed in terms of extremal lengths of simple closed curves \cite{KerckhoffTeich}, one expects the \GM{} com\-pac\-ti\-fi\-ca\-tion to interact nicely with this metric. This is indeed the case, for it turns out that the \GM{} compactification is isomorphic to the horofunction compactification of the Teichm\"uller metric \cite{LiuSu}. In particular, all Teichm\"uller geodesic rays converge in the \GM{} compactification. There is even an explicit formula for their limits \cite{WalshTeich}. In contrast, Teichm\"uller rays can accumulate onto intervals \cite{Lenzhen,LLR,CMW}, circles \cite{BLMR}, and even higher-dimensional sets \cite{LMR} in the Thurston boundary.

Besides Teichm\"uller geodesics, another family of paths that are used extensively in Teichm\"uller dynamics are the horocycles obtained by shearing half-trans\-la\-tion structures (coming from quadratic differentials) with the matrices \[h_t = \begin{pmatrix} 1 & t \\ 0 & 1\end{pmatrix}\]
for $t\in\RR$.
These are called horocycles because the  Teichm\"uller disk gene\-rated by a quadratic differential $q$ is isometric to the hyperbolic plane, and the path $\{ h_t\,  q \st t \in \RR \}$ traces a horocycle (i.e., a limit of circles whose centers go off to infinity) in this plane. Since horocycles converge in the horofunction compacti\-fi\-cation of the hyperbolic plane (which is the same as the visual compactification) as $t$ tends to $\pm \infty$, it is natural to ask whether they converge in the horofunction compactification of Teichm\"uller space. That is the case whenever the horizontal foliation of $q$ consists of a single cylinder \cite[Theorem 17]{Alberge} or is uniquely ergodic \cite[Theorem 20]{Alberge} \cite[Theorem 1.3]{JiangSu}. However, the goal of this note is to give an example of a horocycle that does not converge in the horofunction (or Gardiner--Masur) compactification. The example is then lifted to all Teichm\"uller spaces of Riemann surfaces of genus $g$ with $p$ punctures such that $3g+p>4$ via covering constructions. 

\begin{thm} \label{thm:main}
In every Teichm\"uller space of complex dimension greater than one, there exists a horocycle which does not converge in the horofunction compactification of the Teichm\"uller metric.
\end{thm}

This has the following immediate consequence,  which was first observed by Miyachi in genus two \cite[Section 8.1]{MiyachiLipschitz} (see also \cite[Section 4.3]{Alberge}).

\begin{cor}
In every Teichm\"uller space of complex dimension greater than one, there exists a Teichm\"uller disk whose isometric inclusion does not extend continuously to the horoboundaries.
\end{cor}

The example underlying \thmref{thm:main} is a horocycle generated by a Jen\-kins--Strebel quadratic differential $q$ with two cylinders on the five-times-punctu\-red sphere. The proof that this horocycle diverges consists in two parts. First, we show that for any $s \in \RR$, the sequence $(h_{s+n}\, q)_{n=1}^\infty$ (obtained by applying successive powers of a Dehn multitwist to $h_{s}\, q$) converges in the Gardiner--Masur boundary and we describe its limit. This is deduced from a more general criterion for convergence along mapping class group orbits (\lemref{lem:criterion}), which we also use to reprove that every horocycle or earthquake directed by a simple closed curve converges (\cite[Theorem 20]{Alberge} and \cite[Corollary 3.2]{JiangSu}). The second step is to show that the limit of this sequence depends on $s$. To prove this, it suffices to check that the extremal length of a certain simple closed curve $\alpha$ is not constant along the horocycle $h_{s} \, q$. In \lemref{lem:crux}, we show that the extremal length of $\alpha$ attains a strict local maximum at $s=0$ (and hence at all the integers by pe\-rio\-di\-city). This contrasts with the convexity of hyperbolic length along earthquakes \cite[Theorem 1]{KerckhoffNielsen} and complements the existence of local maxima for extremal length along Teichm\"uller geodesics \cite{nonconvex}.

As argued in \cite{WalshTeich}, the Gardiner--Masur compactification of Teichm\"uller space is best suited for problems concerning the conformal structure of surfaces whereas the Thurston compac\-ti\-fication is tailored for doing hyperbolic geometry. There is a well-known dictionary between the two worlds, partially given in Table \ref{table} below (see \cite[p.33]{Papadopoulos} for an extended version).

\begin{table}[htp] 
\caption{A dictionary between the conformal and hyperbolic aspects of Teichm\"uller theory} \label{table}
\centering
\begin{tabular}{ |c|c| } 
 \hline
Conformal & Hyperbolic \\
 \hline
 \hline
Quasiconformal homeomorpisms & Lipschitz maps \\
 \hline
Teichm\"uller metric & Thurston metric \\
 \hline
Extremal length & Hyperbolic length \\
 \hline
Measured foliations & Geodesic laminations \\
\hline
Gardiner--Masur compactification & Thurston compactification \\
\hline
Teichm\"uller rays & Stretch paths \\
\hline
Horocycles & Earthquakes \\
\hline
\end{tabular}
\end{table}

The analogy between the two points of view is reinforced by the fact that the Thurston compactification is isomorphic to the horofunction compac\-ti\-fication of the Thurston metric \cite{WalshThurs}. Since Teichm\"uller rays, stretch paths, and earthquakes all converge in their respective `compatible' compactification, it is somewhat surprising that horocycles do not.

\begin{acknowledgements}
I thank John Smillie for encouraging me to think about this problem and for useful discussions.
\end{acknowledgements}

\section{Definitions}

We begin by recalling some definitions and results needed in the sequel.

\subsection*{Teichm\"uller space} Let $S$ be an oriented surface with finitely generated fundamental group. The \emph{Teichm\"uller space} $\teich(S)$ is the set of equivalence classes $[(X,f)]$ of pairs $(X,f)$ where $X$ is a closed Riemann surface minus a finite set, the \emph{marking} $f:S \to X$ is an orientation-preserving homeomorphism, and two pairs $(X,f)$ and $(Y,g)$ are equivalent if the change of marking $g\circ f^{-1}:X \to Y$ is homotopic to a biholomorphism. 

The \emph{Teichm\"uller distance} between two points $[(X,f)]$ and $[(Y,g)]$ is $\frac12\log K$ where $K\geq 1$ is the smallest real number such that $g\circ f^{-1}$ is homotopic to a $K$-quasiconformal homeomorphism.

We will usually suppress the marking and the equivalence class from the notation and write $X \in \teich(S)$ to mean $[(X,f)]$ for some marking $f$.

\subsection*{Mapping class group}

The \emph{mapping class group} $\mcg(S)$ of a surface $S$ is the group of homotopy classes of orientation-preserving homeomorphisms of $S$ onto itself. This group acts on the left on homotopy classes of objects on $S$ (such as closed curves) and on the right on Teichm\"uller space by pre-composing the marking. That is, if $[(X,f)]\in \teich(S)$ and $[\phi] \in \mcg(S)$ then
\[
[(X,f)]\cdot[\phi] := [(X,f\circ \phi)].
\]
We will write $X \cdot \phi$ in lieu of the above if the marking is implicit.

\subsection*{Quadratic differentials}

A \emph{quadratic differential} on a Riemann surface $X$ is a map $q: TX \to \CC$ such that $q(\lambda v) = \lambda^2 q(v)$ for every $\lambda \in \CC$ and $v \in TX$. We only consider quadratic diferentials that are holomorphic and whose area $\int_X |q|$ is finite. A \emph{horizontal trajectory} for $q$ is a maximal smooth path $\gamma:\RR \to X$ such that  $q(\gamma'(t))>0$ for every $t \in \RR$.

\subsection*{Extremal length}

Let $\scc(S)$ be the set of homotopy classes of essential (not homotopic to a point or a puncture) simple (embedded) closed curves in $S$. We assume that $S$ is not a sphere with at most $3$ punctures so that $\scc(S)$ is non-empty.

The \emph{extremal length} of $[\gamma] \in \scc(S)$ on $[(X,f)] \in \teich(S)$ is 
\begin{equation} \label{eq:ext}
\ext(\gamma, X):=\sup_{\rho} \frac{\inf_{\alpha \sim f(\gamma)} \ell_\rho(\alpha)^2}{\area(\rho)}
\end{equation}
where $\ell_\rho(\alpha)$ is the length of $\alpha$ with respect to $\rho$ and the supremum is taken over all conformal metrics $\rho$ of finite positive area on $X$. 

Any Riemann surface $A$ with infinite cyclic fundamental group is biholomorphic to a Euclidean cylinder $C$, and the extremal length of either generator of its fundamental group is equal to the ratio of the circumference of $C$ to its height. We denote this number by $\ext(A)$. 

Since extremal length is monotone under conformal embeddings, we can estimate the extremal length of a curve from above using embedded annuli. 
\begin{thm}[Jenkins] \label{thm:Jenkins}
Let $\gamma \in \scc(S)$ and $X \in \teich(S)$, and let $A \subset X$ be an annulus such that the generators of $\pi_1(A)$ are homotopic to $\gamma$. Then
\[
\ext(\gamma, X) \leq \ext(A)
\]
with equality if and only if the pull-back of $dz^2$ under any biholomorphism from $A$ to a Euclidean cylinder $\{ z \in \CC : 0 < \im z < m\}/\ZZ$
extends to a quadratic differential on $X$. Such an extremal annulus always exists, and is unique if $S$ is not a torus.
\end{thm}

If $q$ is the quadratic differential alluded to in the theorem, then the supremum in \eqref{eq:ext} is realized only for the conformal metric $\sqrt{|q|}$ and its scalar multiples. Both statements can be generalized in three different ways to multicurves and collections of disjoint annuli \cite{Jenkins,Strebel,Renelt}.

The notion of extremal length can be extended from $\scc(S)$ to the set $\mf(S)$ of equivalence classes of \emph{measured foliations} on $S$ (see \cite{FLP} for the definition) by setting
\[
\ext(F, X) : = \int_X |q_F|
\]
for all $F \in \mf(S)$ and $X \in \teich(S)$, where $q_F$ is the unique quadratic differential on $X$ whose horizontal foliation is measure-equivalent to $F$ \cite{HubbardMasur,KerckhoffTeich}. The extremal length function $\ext: \mf(S) \times \teich(S) \to \RR$  is continuous, as well as homogeneous of degree $2$ in the first variable.

Although we will not use this here, we mention in passing that the Teich\-m\"uller distance can be recovered from extremal lengths via Kerckhoff's formula \cite{KerckhoffTeich}:
\[
d(X,Y) = \frac12 \log\left( \sup_{\gamma \in \scc(S)}  \frac{\ext(\gamma,Y)}{\ext(\gamma,X)}\right) \quad \text{for all }X,Y \in \teich(S).
\]

\subsection*{The \GM{} compactification}

Let $\RR_{\geq 0} := [0,\infty)$ be the set of non-negative real numbers. The projective space $\PP\left(\RR_{\geq 0}^{\scc(S)}\right)$ is the quotient of $\RR_{\geq 0}^{\scc(S)} \setminus \{ 0 \}$ by the action of $\RR_{>0}$ by multiplication. It is given the quotient topology inherited from the product topology on $\RR_{\geq 0}^{\scc(S)}$.

Gardiner and Masur \cite[Section 6]{GardinerMasur} showed that the map 
\[
\begin{array}{ccccl}
\Phi & :& \teich(S) &\to& \PP\left(\RR_{\geq 0}^{\scc(S)}\right) \\
&& X & \mapsto &\left[ \, \ext^{1/2}(\gamma,X) \,\right]_{\gamma \in \scc(S)}
\end{array}
\]
is an embedding, and that the closure of its image is compact. The \emph{Gardiner--Masur compactification} of $\teich(S)$ is the set $\overline{\Phi(X)} \subset \PP\left(\RR_{\geq 0}^{\scc(S)}\right)$, which we also denote by $\gm(S)$. A sequence $(X_n)_{n=1}^\infty \subset \teich(S)$ \emph{converges} to a projective vector $v \in \gm(S)$ if $\Phi(X_n) \to v$ as $n \to \infty$.  Besides in \cite{GardinerMasur}, this compactification has been studied in \cite{Miyachi1,Miyachi2,MiyachiSurvey,MiyachiLipschitz,MiyachiUnification,LiuSu,Alberge,JiangSu,WalshTeich}.

\subsection*{The horofunction compactification}

The \emph{horofunction compactification} of a proper metric space $(M,d)$ is the set of all locally uniform limits of functions $M \to \RR$  of the form \[y \mapsto d(y,x_n) - d(x_n,b)\] where
$b \in M$ is a fixed basepoint, and $(x_n)_{n=1}^\infty$ ranges over all sequences in $M$. Its elements are called \emph{horofunctions}, and their level sets are called \emph{horospheres}.

Liu and Su \cite{LiuSu} proved that the horofunction compactification of $\teich(S)$ equipped with the Teichm\"uller metric is isomorphic to the Gardiner--Masur compactification. We will mostly work with the Gardiner--Masur formulation except in \secref{sec:lift} where the horofunction point of view simplifies things.

\subsection*{Horocycles lie on horospheres}
It is interesting to note that any horocycle $h_t\, q$ obtained by shearing a quadratic differential $q$ travels along some horosphere, namely, a level set of the function
\begin{equation} \label{eq:levelset}
X \mapsto -\frac12 \log \ext(F,X)
\end{equation}
where $F$ is the horizontal foliation of $q$. Indeed, the horizontal foliation of $h_t\, q$ and its area does not depend on $t$, so that if $X_t$ denotes the underlying Riemann surface, then
\[
\ext(F, X_t) = \int_{X_t} |h_t\, q|  =  \int_{X_0} |q| = \ext(F, X_0)
\]
for all $t \in \RR$.
That \eqref{eq:levelset} defines a horofunction (in fact a Busemann function) when $F$ is a simple closed curve follows from the proof of \cite[Lemma 3.3]{nonconvex}. The density of weighted simple closed curves in $\mf(S)$ and the continuity of extremal length then imply that \eqref{eq:levelset} defines a horofunction for any $F \in \mf(S)$, though not a Busemann function (the limit of an almost geodesic ray) in general. 

The horocycle $h_t\, q$ also lies on a level set of the horofunction obtained as the forward limit of the Teichm\"uller geodesic $g_s\, q$ where
\[
g_s =\begin{pmatrix} e^{s} & 0 \\ 0 & e^{-s}  \end{pmatrix}.
\]
The resulting Busemann function can be determined by combining Liu and Su's isomorphism \cite[Section 5]{LiuSu} between the horofunction and Gardiner--Masur compactifications and Walsh's formula \cite[Corollary 1]{WalshTeich} for the limit of $g_s\, q$ in $\gm(S)$, though this gives a rather convoluted expression.

Our point is that the horocycle $\{ h_t\, q : t \in \RR\}$ is far from an arbitrary path---it is the intersection of one or more horospheres with a complex geodesic---yet it can still accumulate onto a non-trivial continuum in the horofunction boundary, as we will see in \secref{sec:example}.

\section{Convergence along mapping class group orbits}

Our first result is a sufficient criterion for a sequence in the mapping class group orbit of a point $X \in\teich(S)$ to converge in the \GM{} compactification.  The idea behind this criterion was already exploited in \cite[Theorem 7.2]{GardinerMasur} and \cite[Proposition 4.1]{JiangSu}.

\begin{lem}  \label{lem:criterion}
Let $(\phi_n)_{n=1}^\infty \subset \mcg(S)$ be a sequence of mapping classes. Suppose that there exists a sequence $(c_n)_{n=1}^\infty$ of positive real numbers and a non-zero function $f: \scc(S) \to \mf(S)$ such that $c_n \phi_n(\gamma)\to f(\gamma)$ as $n \to \infty$ for every $\gamma \in \scc(S)$. Then for every $X \in \teich(S)$, the sequence $X\cdot \phi_n$ converges to the projective vector
\[
\left[\ \ext^{1/2}(f(\gamma),X) \,\right]_{\gamma \in \scc(S)}
\]
 in the Gardiner--Masur compactification $\gm(S)$ as $n\to \infty$. 
\end{lem}
\begin{proof}
Let $\gamma \in \scc(S)$. By definition of the mapping class group action on Teichm\"uller space, we have $\ext(\gamma, X\cdot \phi_n) = \ext(\phi_n(\gamma), X)$ for every $n\geq 1$.
It follows that
\[
c_n^2 \ext(\gamma, X\cdot \phi_n) = c_n^2\ext(\phi_n(\gamma), X)= \ext(c_n \phi_n(\gamma), X) \rightarrow  \ext(f(\gamma), X)
\]
as $n \to \infty$, by homogeneity and continuity of extremal length on $\mf(S)$. Thus $X\cdot \phi_n$ converges to the stipulated limit in $\gm(S)$ as $n\to \infty$.
\end{proof}

A \emph{Dehn multitwist} is a product $\tau_1^{n_1} \circ \cdots \circ \tau_k^{n_k}$  of non-zero integer powers $n_j$ of Dehn twists $\tau_j$ about the components $\alpha_j$ of a multicurve on a surface. We will apply the above criterion when $\phi_n=\phi^n$ is a sequence of powers of a Dehn multitwist $\phi$. In order to apply the criterion, we first need to understand the effect of $\phi$ and its powers on simple closed curves. The follo\-wing estimate from \cite[Lemma 4.2]{Ivanov} generali\-zing \cite[Proposition A.1]{FLP} is used to determine the projective limit of $\phi^{n}(\gamma)$ as $n \to \infty$ for any curve $\gamma \in \scc(S)$.

\begin{lem}[Ivanov] \label{lem:ivanov}
Let $\tau=\tau_1^{n_1} \circ \cdots \circ \tau_k^{n_k}$ be a Dehn multitwist about a multicurve $\{\alpha_1,\ldots,\alpha_k\}$ in a surface $S$. Then for any two essential simple closed curves $\gamma, \beta \in \scc(S)$ we have
\begin{align*}
 \sum_{j=1}^k (|n_j|-2) i(\gamma , \alpha_j)i(\alpha_j, \beta) - i(\gamma,\beta) & \leq i(\tau(\gamma), \beta) \\ & \leq  \sum_{j=1}^k |n_j| i(\gamma , \alpha_j)i(\alpha_j, \beta) + i(\gamma, \beta)
\end{align*}
where $i$ is the geometric intersection number.
\end{lem}

In particular, for fixed curves $\gamma$ and $\beta$, the difference between $i(\tau(\gamma), \beta)$ and $\sum_{j=1}^k |n_j| i(\gamma , \alpha_j)i(\alpha_j, \beta)$ is bounded independently of the powers $n_j$. We apply this to successive powers of a fixed Dehn multitwist $\phi$.

\begin{cor} \label{cor:limit_dehn_twist}
Let $\phi= \tau_1^{n_1} \circ \cdots \circ \tau_k^{n_k}$ be a Dehn multitwist about a multicurve $\{\alpha_1,\ldots,\alpha_k\}$ in a surface $S$ and let $\gamma\in \scc(S)$ be any simple closed curve. Then $\phi^{n}(\gamma) /n$  converges to the weighted multicurve
\[
 \sum_{j=1}^k |n_j| i(\gamma , \alpha_j) \alpha_j
\]
in $\mf(S)$ as $n \to \infty$.
\end{cor}
\begin{proof}
Let $\beta \in \scc(S)$ be any simple closed curve. We need to show that 
\[
i(\phi^{n}(\gamma) /n, \beta)=\frac{1}{n}i(\phi^{n}(\gamma),\beta) \quad \text{converges to} \quad  \sum_{j=1}^k |n_j| i(\gamma , \alpha_j) i(\alpha_j,\beta)
\]
as $n \to \infty$. This follows from \lemref{lem:ivanov} applied to $\tau=\phi^n$ since the error terms tend to zero after dividing by $n$.
\end{proof}

Now that we know the projective limit of $\phi^{n}(\gamma)$ as $n \to \infty$, we can apply our criterion to deduce that the orbit of a point in Teichm\"uller space under the cyclic group generated by a Dehn multitwist converges in the Gardiner--Masur compactification (in either direction). Furthermore, we get a formula for the limit.

\begin{cor} \label{cor:dehn}
Let $\phi= \tau_1^{n_1} \circ \cdots \circ \tau_k^{n_k}$ be a Dehn multitwist about a multicurve $\{\alpha_1,\ldots,\alpha_k\}$ in a surface $S$ and let $X \in \teich(S)$. Then the sequence $X \cdot \phi^n$ converges to 
\[
\left[ \ext^{1/2} \left( \sum_{j=1}^k |n_j| i(\gamma , \alpha_j) \alpha_j,X\right) \right]_{\gamma \in \scc(S)}
\]
 in the Gardiner--Masur compactification $\gm(S)$ as $n\to \infty$.
\end{cor}
\begin{proof}
\corref{cor:limit_dehn_twist} shows that the hypotheses of \lemref{lem:criterion} are satisfied for $\phi_n = \phi^n$, $c_n = 1/n$, and $f(\gamma) =  \sum_{j=1}^k  |n_j|i(\gamma , \alpha_j) \alpha_j$, from which the conclusion follows.
\end{proof}

Observe that the limit given in \lemref{lem:criterion} or \corref{cor:dehn} appears to depend on the initial point $X$. However, it may happen that for every surface $Y \in \teich(S)$ there is a constant $c>0$ such that \[\ext(f(\gamma), Y) =  c\ext(f(\gamma), X)\]  for all $\gamma \in \scc(S)$, which results in the projective vector 
\[
\left[\, \ext^{1/2}(f(\gamma),X) \,\right]_{\gamma \in \scc(S)}
\] being independent of $X$. This occurs if the image of $f$ is contained in a single ray, for example if $\phi$ is a Dehn twist about a single curve $\alpha\in\scc(S)$. In that case, $f(\gamma)=i(\gamma, \alpha) \alpha$ and $X \cdot \phi^n$ converges to
\[
\left[\, \ext^{1/2}(i(\gamma, \alpha) \alpha,X) \,\right]_{\gamma \in \scc(S)} = \left[\, i(\gamma, \alpha) \ext^{1/2}( \alpha,X) \,\right]_{\gamma \in \scc(S)} = \left[\, i(\gamma, \alpha) \,\right]_{\gamma \in \scc(S)}
\]
as $n \to \infty$. By varying $\alpha$ and taking limits, it follows that $\left[\, i(\gamma, F) \,\right]_{\gamma \in \scc(S)}$ belongs to $\gm(S)$ for any measured foliation $F \in \mf(S)$. That is, the Gardiner--Masur boundary contains the Thurston boundary of projective measured foliations (both are subsets of the same projective space), as was observed in \cite[Theorem 7.1]{GardinerMasur}. In the next section, we will give an example of a Dehn multitwist $\phi$ where the limit of $X \cdot \phi^n$ actually depends on $X$. 

\lemref{lem:criterion} can of course be applied to other sequences of mapping classes. For instance, if $\phi$ is a pseudo-Anosov with horizontal and vertical foliations $\calH$ and $\calV$ and stretch factor $\lambda>1$, then for any $\gamma\in \scc(S)$ the sequence $\lambda^{-n}\phi^{n}(\gamma)$ converges to $\frac{i(\gamma,\calV)}{i(\calH,\calV)}\, \calH$ as $n \to \infty$. In this case, $X\cdot\phi^n$ converges to
\begin{align*}
\left[\, \ext^{1/2}\left(\frac{i(\gamma,\calV)}{i(\calH,\calV)}\, \calH,X\right) \,\right]_{\gamma \in \scc(S)} &= \left[ \frac{i(\gamma,\calV)}{i(\calH,\calV)}\, \ext^{1/2}(\calH,X) \,\right]_{\gamma \in \scc(S)} \\ &= \left[\, i(\gamma,\calV)\,\right]_{\gamma \in \scc(S)}
\end{align*}
as $n \to \infty$, which is manifestly independent of $X$. 

When this phenomenon happens, that is, when the limit in \lemref{lem:criterion} is independent of $X$, we can promote convergence along sequences to convergence along paths.

\begin{prop}
Suppose that $\phi \in \mcg(S)$ is such that the sequence defined by $\phi_n:=\phi^n$ satisfies the hypotheses of \lemref{lem:criterion} and is such that the limit of $X\cdot \phi^n$ as $n \to \infty$ does not depend on $X$. Then every continuous path $\omega: \RR \to \teich(S)$ for which there is a $T>0$ such that $\omega(t+T) = \omega(t) \cdot \phi$ for all $t \in \RR$ converges to the same limit as $t \to \infty$.
\end{prop}

\begin{proof}
\lemref{lem:criterion} generalizes easily to sequences $(X_n)_{n=1}^\infty$ such that $X_n \cdot \phi_n^{-1}$ converges to some $X \in \teich(S)$ as $n \to \infty$, with the conclusion that $X_n$ converges to 
\[
\left[\ \ext^{1/2}(f(\gamma),X) \,\right]_{\gamma \in \scc(S)}
\]
as $n \to \infty$. By hypothesis, this limit $L$ is independent of $X$.

To prove the result, it suffices to show that any sequence $(t_n)_{n=1}^\infty \subset \RR$ tending to infinity admits a subsequence such that $\omega(t_{n_k}) \to L$ as $k \to \infty$. For each $n \geq 1$, let $m_n \in \ZZ$ be such that $s_n :=t_n-m_nT$ belongs to the interval $[0, T]$. Then $\omega(t_n) \cdot \phi^{-m_n} = \omega(s_n) $ for every $n\geq 1$. Since $[0,T]$ is compact and $\omega$ is continuous, there is an $s \in [0,T]$ and a subsequence such that $s_{n_k} \to s$ and $ \omega(s_{n_k}) \to \omega(s)$ as $k \to \infty$. By the previous paragraph, we get that $\omega(t_{n_k}) \to L$ as $k \to \infty$. 
\end{proof}

If $\phi$  is a Dehn twist about a curve $\alpha \in \scc(S)$, then the above proposition implies that every horocycle directed by a Jenkins--Strebel differential with a single cylinder homotopic to $\alpha$ and every earthquake directed by $\alpha$ converges to $\left[\, i(\gamma, \alpha) \,\right]_{\gamma \in \scc(S)}$. This recovers \cite[Theorem 17]{Alberge} and \cite[Corollary 3.2]{JiangSu} respectively. If $\phi$ is a pseudo-Anosov, then we get that the axis of $\phi$ converges to $\left[\, i(\gamma,\calV)\,\right]_{\gamma \in \scc(S)}$ in the forward direction, where $\calV$ is the vertical foliation. Since the foliations of a pseudo-Anosov are uniquely ergodic, this also follows from \cite[Corollary 2]{Miyachi2} or \cite[Corollary 1]{WalshTeich}. On the other hand, the result applies to all $\phi$-invariant paths. 

\section{A divergent horocycle} \label{sec:example}

Let $S^1=\RR/\ZZ$ and let $C=S^1 \times [-1,1]$. Seal the top and bottom of $C$ shut via the relation $(x,y) \sim (-x,y)$ for all $(x,y)\in S^1 \times\{-1,1\}$ to create a pillowcase, and remove the four corners $(0,\pm 1)$ and $(1/2,\pm 1)$ as well as $(0,0)$ to get a five-times-punctured sphere $X$ equipped with a quadratic differential $q$ coming from the differential $dz^2$ on $C$. Applying the horocycle flow $h_t$ to $q$ results in a twisted punctured pillowcase denoted $X_t$. 

\begin{prop} \label{prop:diverges}
The horocycle $t \mapsto X_t$ defined above does not converge in the Gardiner--Masur compactification as $t \to \infty$.
\end{prop}

The proof proceeds by finding distinct limits of sequences going to infinity along the path. This is sufficient since the projective space $\PP\left(\RR_{\geq 0}^{\scc(X)}\right)$ is Hausdorff, so that limits are unique when they exist.

 Consider the homeomorphism $\phi$ obtained by applying a right Dehn twist about each of the horizontal curves $\alpha$ and $\beta$ at heights $-1/2$ and $1/2$ in $X$. Then $X_{s+n} = X_s \cdot \phi^n$ for every $s \in \RR$ and $n \in \ZZ$. Indeed, $\phi$ can be realized by the matrix \[h_{1} = \begin{pmatrix} 1 & 1 \\ 0 & 1 \end{pmatrix}\] and the marking for $X_{s+n}$ is given by $h_{s+n} = h_s \circ h_n = h_s \circ \phi^{n}$ (recall that mapping classes act on Teichm\"uller space on the right by pre-composing the marking, while the horocycle flow acts on the left by post-composing charts with matrices). By \corref{cor:dehn}, the sequence $X_{s+n}$ converges to the projective vector
\begin{equation} \label{eq:limit}
v_s = \left[ \ext^{1/2} \left( i(\gamma , \alpha) \alpha  + i(\gamma , \beta) \beta ,X_s\right) \right]_{\gamma \in \scc(S)}
\end{equation}
as $n \to \infty$.
We will show that the limit $v_s$ is not a constant function of $s$, which implies \propref{prop:diverges}.

We first observe that $\ext\left( \alpha  + \beta ,X_s\right)$ is constant equal to $2$ since the qua\-dra\-tic differential $h_s q$ has horizontal foliation $\alpha+\beta$ and area $2$. Suppose on the other hand that $\ext\left( \alpha,  X_s\right)$ is not constant in $s$. Then the projective vector $v_s$ depends on $s$, since we can find simple closed curves $\eta$ and $\nu$ in $X$ such that 
\[
i(\eta , \alpha) \alpha  + i(\eta , \beta) \beta = 2\alpha \quad \text{and} \quad i(\nu , \alpha) \alpha  + i(\nu , \beta) \beta  = 2(\alpha+\beta)\] 
(see \figref{fig:curves}). The ratio of the corresponding entries in $v_s$ is then 
\[
\frac{\ext^{1/2} \left( \alpha  ,X_s \right)}{\ext^{1/2} \left( \alpha + \beta,X_s\right)} = \frac{\ext^{1/2} \left( \alpha ,X_s \right)}{\sqrt{2}},
\] which is non-constant by hypothesis.

\begin{figure} 
\centering
\includegraphics[scale=.9]{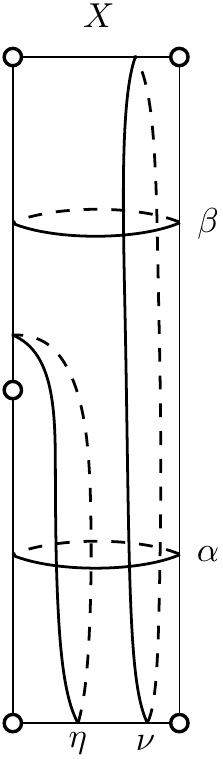}
\caption{The punctured pillowcase $X$ with some curves on it.} \label{fig:curves}
\end{figure}

We have thus reduced \propref{prop:diverges} to showing that $\ext\left( \alpha ,X_s\right)$ is not constant in $s$, which is our next result.

\begin{lem} \label{lem:crux}
The function $s \mapsto \ext\left( \alpha,  X_s\right)$ attains a strict local maximum at zero.
\end{lem}

\begin{proof} The curve $\alpha$ is invariant under the anti-conformal involution  of $X$ given by $(x,y) \mapsto (-x,y)$ in cylinder coordinates. The embedded annulus $A \subset X$ realizing the extremal length of $\alpha$ is also invariant under that symmetry since it is unique. It follows that $A$ is disjoint from from the top and bottom edges of the punctured pillowcase $X$. In other words, $A$ is contained in the punctured cyclinder $B=S^1 \times (-1,1) \setminus \{(0,0)\}$. The latter embeds conformally in $X_s$ for every $s\in \RR$.  Indeed, $B$ is clearly invariant under the horocycle flow. It is the identifications along its boundary that change to $(x,1) \sim (2s-x,1)$ and $(x,-1) \sim (-2s-x,-1)$ in order to obtain $X_s$ (after puncturing at the folding points).

Since $A$ embeds conformally in $B$ and then in $X_s$ (in the same homotopy class as $\alpha$), \thmref{thm:Jenkins} tells us that 
\[
\ext\left( \alpha ,X_s\right) \leq \ext(A) = \ext\left( \alpha ,X\right)
\] with equality if and only if the standard quadratic differential $\psi$ on $A$ (which pulls back to $dz^2$ in cylindrical coordinates) extends to a quadratic differential on $X_s$. In turn, this happens if and only if the gluing used to obtain $X_s$ from $B$ is isometric with respect to $\psi$.

\begin{figure} 
\centering
\includegraphics[scale=.9]{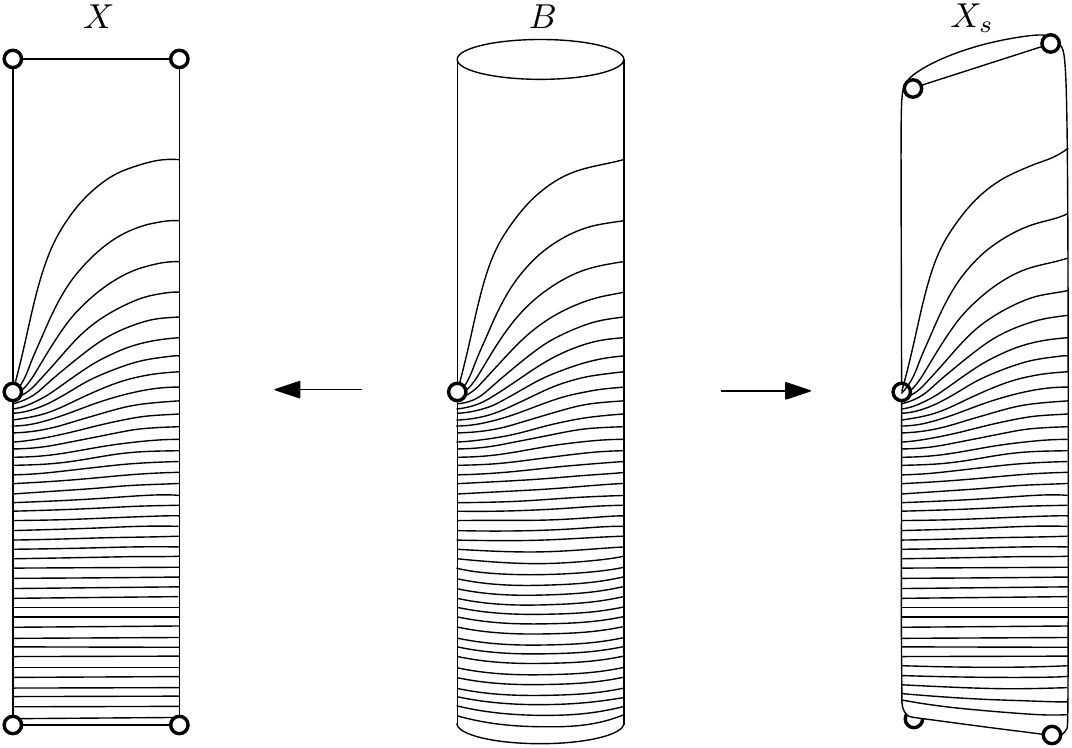}
\caption{The quadratic differential realizing the extremal length of $\alpha$ on the punctured cylinder $B$ extends to a quadratic differential on $X$ but not on $X_s$ for any small $s\neq 0$.} \label{fig:diff}
\end{figure}

The identifications $(x,1) \sim (2s-x,1)$ and $(x,-1) \sim (-2s-x,-1)$ on the top and bottom boundaries are of course isometries with respect to the quadratic differential $dz^2$ on $B$, but we claim that there is a neighborhood $N$ of $0$ in $\RR$ such that they are not isometries with respect to $\psi$ for any $s \in N\setminus \{ 0 \}$. The main reason for this is that $\psi$ is not rotationally symmetric.

We focus on the bottom boundary. Let $h : S^1 \times (0,m) \to A$ be a biholomorphism between a Euclidean cylinder and $A$ chosen to be equivariant under the symmetry $(x,y) \mapsto (-x,y)$, where the $(x,y)$-coordinates in the target are those from $B$. This map $h$ followed by the inclusion $\iota : A \hookrightarrow B$ extends to an odd analytic diffeomorphism $g$ bet\-ween the bottom circles of $S^1 \times (0,m)$ and $B$. Since $g$ is odd, we have $g''(0)=0$. On the other hand, $g''$ is not constant equal to zero since $\iota\circ h$ is not affine.  By the identity principle, there is a neighborhood $U$ of $0$ in $S^1 \times \{0 \}$ such that $g''(x) \neq 0$ for every $x \in U\setminus \{0\}$. We may assume that $U$ is connected, in which case we deduce that $g'(u)\neq g'(v)$ for every $u,v \in U$ such that $u<v<0$ or $0<u<v$ by Rolle's theorem.

In the cylinder coordinate $S^1 \times (0,m)$, the pull-back $h^*\psi$ becomes $dz^2$. Thus the line element $\sqrt{|\psi|}$ induced by $\psi$ on the bottom circle of $B$ is the push-forward of the Euclidean line element $|dx|$ by $g$.
If the identification $x \sim -2s-x$ on the bottom of $B$ is an isometry with respect to $\psi$, then $(g^{-1})'(x) = (g^{-1})'(-2s - x)$ for every $x$. If $s>0$ is sufficiently small and $-2s<x<0$ then $u=g^{-1}(x)$ and $v=g^{-1}(-2s-x)$ are both in $U$, to the left of $0$, and satisfy \[g'(u)=\frac{1}{(g^{-1})'(x)}=\frac{1}{(g^{-1})'(-2s - x)}=g'(v),\] contradicting the previous paragraph. Similarly, if $s$ is negative and sufficiently close to zero then we can find a pair of points $u,v \in U$ with $0<u<v$ such that $g'(u)=g'(v)$.    This contradiction implies our claim that there is a neighborhood $N$ of $0$ in $\RR$ such that $\psi$ does not extend to a quadratic differential on $X_s$ for any $s \in N \setminus \{0\}$. We conclude that $\ext\left( \alpha ,X_s\right) < \ext\left( \alpha ,X_0\right)$ for every $s\in N \setminus \{0\}$.
\end{proof}

The above proof shows that $\ext\left( \alpha ,X_s\right)  \leq \ext\left( \alpha ,X_0\right)$ for all $s \in \RR$. The extremal length at $s=0$ can be computed using Schwarz--Christoffel transformations, and is approximately $0.8196442$. The horizontal trajectories of the corresponding quadratic differential are sketched in \figref{fig:diff}.

As stated in the introduction, the function $s \mapsto \ext(\alpha,X_s)$ is periodic since 
\[\ext\left( \alpha,  X_{s+n}\right) =  \ext\left( \alpha,  X_s \cdot \phi^n\right) =  \ext\left( \phi^{n}(\alpha),  X_s \right)= \ext\left( \alpha,  X_s \right)
\]
for every $s\in \RR$ and $n \in \ZZ$. Hence it has strict local maxima at all the integers. 

By a similar reasoning, the projective vector $v_s$ from \eqnref{eq:limit} is $\ZZ$-periodic (and in fact $\frac12\ZZ$-periodic since $\phi$ has a square root preserving both $\alpha$ and $\beta$) and invariant under $s \mapsto -s$.  We think that $s\mapsto v_s$ is injective on $[0,1/4]$ so that the horocycle $t \mapsto X_t$ accumulates onto an interval in the Gardiner--Masur boundary. It would be interesting to find examples with larger limit sets.

\section{Lifting the example to higher complexity} \label{sec:lift}

Let $S_{g,p}$ be an oriented surface of genus $g$ with $p$ punctures. The following lemma is taken from \cite[Lemma 7.1]{Gekhtman}. 

\begin{lem}[Gekhtman--Markovic] \label{lem:lift}
 If $3g-3+p>1$, then there is a branched cover $\overline{S_{g,p}} \to \overline{S_{0,5}}$ that branches at all pre-images of marked points that are not marked and induces an isometric embedding $\teich(S_{0,5}) \hookrightarrow \teich(S_{g,p})$. 
\end{lem}

We use this to export the example from \propref{prop:diverges} to all Teichm\"uller spaces $\teich(S_{g,p})$ of complex dimension $3g-3+p>1$. We explain how this works from the horofunction point of view as well as from the Gardiner--Masur one.

\begin{proof}[First proof \thmref{thm:main}]
Let $t \mapsto X_t$ be the divergent horocycle in $S_{0,5}$ constructed in \propref{prop:diverges} and let  $\iota: \teich(S_{0,5}) \to \teich(S_{g,p})$ be any isometric embedding induced by a branched cover (\lemref{lem:lift}). Then $t \mapsto \iota(X_t)$ is a horocycle in $\teich(S_{g,p})$ since the $\SL(2,\RR)$-action on quadratic differentials commutes with the pull-back by the branched cover.

 Let $b \in \teich(S_{0,5})$ be any basepoint. We take $\iota(b)$ as the basepoint for the horofunction compactification of $\teich(S_{g,p})$ (the choice of basepoint only changes horofunctions by an additive constant). Suppose that $\iota(X_t)$ converges to a horofunction $h:\teich(S_{g,p}) \to \RR$ as $t \to \infty$. Then $X_t$ converges to the function $h \circ \iota$ as $t \to \infty$, a contradiction. Thus $\iota(X_t)$ diverges.
\end{proof}

\begin{proof}[Second proof \thmref{thm:main}]
Let $\pi:\overline{S_{g,p}} \to \overline{S_{0,5}}$ be a branched cover and let $\iota: \teich(S_{0,5}) \to \teich(S_{g,p})$ be the induced isometric embedding. For any simple closed curve $\gamma \in \scc(S_{0,5})$ and any $X \in \teich(S_{0,5})$, we have the identity
\[
\ext(\pi^{-1}(\gamma), \iota(X)) = d \cdot \ext(\gamma,X)
\]
where $d$ is the degree of $\pi$. Indeed, if $\theta$ is the quadratic differential on $X$ whose horizontal foliation is measure-equivalent to $\gamma$, then the horizontal foliation of  the pull-back differential $\pi^*\theta$ is measure-equivalent to $\pi^{-1}(\gamma)$ and the area of $\pi^*\theta$ is $d$ times that of $\theta$.

Let $t \mapsto X_t$ be the divergent horocycle from \secref{sec:example} directed by the quadratic differential $q$, with $\alpha, \beta, \eta, \nu \subset X$ the same curves as in \figref{fig:curves}. Then there exists some $m \in \NN$ such that for every $s \in \RR$ and $n \in \ZZ$ we have $\iota(X_{s+mn}) = \iota(X_{s}) \cdot \psi^n$ where $\psi$ is a Dehn multitwist about $\pi^{-1}(\alpha \cup \beta)$. Indeed, each component $c$ of $\pi^{-1}(\alpha \cup \beta)$  covers either $\alpha$ or $\beta$ with some degree $d_c \in \NN$, which may vary from one component to another. Thus, each cylindrical component of $\pi^*q$ corresponding to a curve $c$ has height $1$ and circumference $d_c$. If $m$ is the least common multiple of the degrees $d_c$, then the matrix $h_m$ performs a right Dehn twist to the power $m/d_c$ about each component $c$ of $\pi^{-1}(\alpha \cup \beta)$.

 In particular, the sequence $\iota(X_{s+mn})$ converges to some limit $w_s$ in the Gardiner--Masur compactification as $n \to \infty$ (\corref{cor:dehn}), but the  limit depends on $s$. Indeed, $\ext(\pi^{-1}(\alpha \cup \beta), \iota(X_s))$ is constant while 
\[
\ext(\pi^{-1}(\alpha), \iota(X_s)) = d \cdot \ext(\alpha, X_s)
\]
 is not by \lemref{lem:crux}. The only difference with the proof of \propref{prop:diverges} is that here $\pi^{-1}(\eta)$ and $\pi^{-1}(\nu)$ are not necessarily simple closed curves in $S_{g,p}$ (they might not be connected). However, the map $f: \scc(S_{g,p}) \to \mf(S_{g,p})$ defined by
\[
f(\gamma) = \sum_{c \subset \pi^{-1}(\alpha \cup \beta)} \frac{m}{d_c}\, i(\gamma ,  c) \,c
\]
(where the sum is over connected components) extends continuously to the space of measured foliations $\mf(S_{g,p})$. By fixing a small $s\neq 0$ such that $\ext(\alpha, X_s) \neq  \ext(\alpha, X_0)$ and by approximating $\pi^{-1}(\eta)$ and $\pi^{-1}(\nu)$ with simple closed curves $\gamma_n, \delta_n \in \scc(S_{g,p})$ we get that
\[
\frac{\ext^{1/2} \left( f(\gamma_n)  ,\iota(X_s) \right)}{\ext^{1/2} \left( f(\delta_n)  ,\iota(X_s) \right)} \neq \frac{\ext^{1/2} \left( f(\gamma_n)  ,\iota(X_0) \right)}{\ext^{1/2} \left( f(\delta_n)  ,\iota(X_0) \right)}
\]
if $n$ is large enough, and hence that $w_s \neq w_0$. Here we are using the fact that
\[
i(\pi^{-1}(\mu),c)= d_c\cdot i(\mu, \pi(c))
\]
for every $\mu \in  \scc(S_{0,5})$ and $c \in  \scc(S_{g,p})$, which implies that
\begin{align*}
f(\pi^{-1}(\eta)) &= \sum_{c \subset \pi^{-1}(\alpha \cup \beta)} \frac{m}{d_c}\, i(\pi^{-1}(\eta) ,  c) \,c \\
&= m \sum_{c \subset \pi^{-1}(\alpha \cup \beta)} i(\eta ,  \pi(c)) \,c \\
&= m \sum_{c \subset \pi^{-1}(\alpha)} i(\eta ,  \alpha) \,c  \\
&= 2m \, \pi^{-1}(\alpha)
\end{align*}
and similarly $f(\pi^{-1}(\nu))= 2m \,\pi^{-1}(\alpha\cup \beta)$.
\end{proof}

\section{Concluding remark}

In \cite[Section 6]{JiangSu}, Jiang and Su conjectured that there exist earthquakes directed by disconnected multicurves that do not converge in the Gardiner--Masur compactification. We agree with this intuition and further believe that all earthquakes and horocycles diverge except for those directed by indecomposable laminations or foliations.

\bibliographystyle{amsalpha}
\bibliography{biblio}

\providecommand{\bysame}{\leavevmode\hbox to3em{\hrulefill}\thinspace}
\providecommand{\MR}{\relax\ifhmode\unskip\space\fi MR }
\providecommand{\MRhref}[2]{%
  \href{http://www.ams.org/mathscinet-getitem?mr=#1}{#2}
}
\providecommand{\href}[2]{#2}
\begin{thebibliography}{CMW19}

\bibitem[Alb16]{Alberge}
V.~Alberge, \emph{Convergence of some horocyclic deformations to the
  {G}ardiner-{M}asur boundary}, Ann. Acad. Sci. Fenn. Math. \textbf{41} (2016),
  no.~1, 439--455.

\bibitem[BLMR]{BLMR}
J.~Brock, C.~Leininger, B.~Modami, and K.~Rafi, \emph{Limit sets of
  {T}eichm\"uller geodesics with minimal nonuniquely ergodic vertical
  foliation, {II}}, to appear in \emph{Journal f\"ur reine und angewandte
  Mathematik}, \href{https://doi.org/10.1515/crelle-2017-0024}{\tt
  https://doi.org/10.1515/crelle-2017-0024}.

\bibitem[CMW19]{CMW}
J.~Chaika, H.~Masur, and M.~Wolf, \emph{Limits in {$\mathcal{PMF}$} of
  {T}eichm\"{u}ller geodesics}, J. Reine Angew. Math. \textbf{747} (2019),
  1--44.

\bibitem[FBR18]{nonconvex}
M.~Fortier~Bourque and K.~Rafi, \emph{Non-convex balls in the {T}eichm\"{u}ller
  metric}, J. Differential Geom. \textbf{110} (2018), no.~3, 379--412.

\bibitem[FLP12]{FLP}
A.~Fathi, F.~Laudenbach, and V.~Po\'{e}naru, \emph{Thurston's work on
  surfaces}, Mathematical Notes, vol.~48, Princeton University Press,
  Princeton, NJ, 2012, Translated from the 1979 French original by Djun M. Kim
  and Dan Margalit.

\bibitem[GM]{Gekhtman}
D.~Gekhtman and V.~Markovic, \emph{Classifying complex geodesics for the
  {C}arath\'eodory metric on low-dimensional {T}eichm\"uller spaces}, to appear
  in \emph{Journal d'Analyse Math\'ematique},
  \href{https://arxiv.org/abs/1711.04722}{\tt arXiv:1711.04722}.

\bibitem[GM91]{GardinerMasur}
F.~Gardiner and H.~Masur, \emph{Extremal length geometry of {T}eichm\"uller
  space}, Complex Variables, Theory and Application: An International Journal
  \textbf{16} (1991), 209--237.

\bibitem[HM79]{HubbardMasur}
J.~Hubbard and H.~Masur, \emph{Quadratic differentials and foliations}, Acta
  Math. \textbf{142} (1979), no.~3-4, 221--274.

\bibitem[Iva92]{Ivanov}
N.V. Ivanov, \emph{Subgroups of {T}eichm\"{u}ller modular groups}, Translations
  of Mathematical Monographs, vol. 115, American Mathematical Society,
  Providence, RI, 1992, Translated from the Russian by E. J. F. Primrose and
  revised by the author.

\bibitem[Jen57]{Jenkins}
J.A. Jenkins, \emph{On the existence of certain general extremal metrics}, Ann.
  of Math. (2) \textbf{66} (1957), 440--453.

\bibitem[JS16]{JiangSu}
M.~Jiang and W.~Su, \emph{Convergence of earthquake and horocycle paths to the
  boundary of {T}eichm\"{u}ller space}, Sci. China Math. \textbf{59} (2016),
  no.~10, 1937--1948.

\bibitem[Ker80]{KerckhoffTeich}
S.P. Kerckhoff, \emph{The asymptotic geometry of {T}eichm\"uller space},
  Topology \textbf{19} (1980), 23--41.

\bibitem[Ker83]{KerckhoffNielsen}
\bysame, \emph{The {N}ielsen realization problem}, Ann. of Math. (2)
  \textbf{117} (1983), no.~2, 235--265.

\bibitem[Len08]{Lenzhen}
A.~Lenzhen, \emph{Teichm\"{u}ller geodesics that do not have a limit in
  {${\mathcal{PMF}}$}}, Geom. Topol. \textbf{12} (2008), no.~1, 177--197.

\bibitem[LLR18]{LLR}
C.~Leininger, A.~Lenzhen, and K.~Rafi, \emph{Limit sets of {T}eichm\"{u}ller
  geodesics with minimal non-uniquely ergodic vertical foliation}, J. Reine
  Angew. Math. \textbf{737} (2018), 1--32.

\bibitem[LMR18]{LMR}
A.~Lenzhen, B.~Modami, and K.~Rafi, \emph{Teichm\"uller geodesics with
  $d$-dimensional limit sets}, Journal of Modern Dynamics \textbf{12} (2018),
  261--283.

\bibitem[LS14]{LiuSu}
L.~Liu and W.~Su, \emph{The horofunction compactification of the
  {T}eichm\"{u}ller metric}, Handbook of {T}eichm\"{u}ller theory. {V}ol. {IV},
  IRMA Lect. Math. Theor. Phys., vol.~19, Eur. Math. Soc., Z\"{u}rich, 2014,
  pp.~355--374.

\bibitem[Miy08]{Miyachi1}
H.~Miyachi, \emph{Teichm\"{u}ller rays and the {G}ardiner-{M}asur boundary of
  {T}eichm\"{u}ller space}, Geom. Dedicata \textbf{137} (2008), 113--141.

\bibitem[Miy13]{Miyachi2}
\bysame, \emph{Teichm\"{u}ller rays and the {G}ardiner-{M}asur boundary of
  {T}eichm\"{u}ller space {II}}, Geom. Dedicata \textbf{162} (2013), 283--304.

\bibitem[Miy14a]{MiyachiSurvey}
\bysame, \emph{Extremal length geometry}, Handbook of {T}eichm\"{u}ller theory.
  {V}ol. {IV}, IRMA Lect. Math. Theor. Phys., vol.~19, Eur. Math. Soc.,
  Z\"{u}rich, 2014, pp.~197--234.

\bibitem[Miy14b]{MiyachiLipschitz}
\bysame, \emph{Lipschitz algebras and compactifications of {T}eichm\"{u}ller
  space}, Handbook of {T}eichm\"{u}ller theory. {V}ol. {IV}, IRMA Lect. Math.
  Theor. Phys., vol.~19, Eur. Math. Soc., Z\"{u}rich, 2014, pp.~375--413.

\bibitem[Miy14c]{MiyachiUnification}
\bysame, \emph{Unification of extremal length geometry on {T}eichm\"{u}ller
  space via intersection number}, Math. Z. \textbf{278} (2014), no.~3-4,
  1065--1095.

\bibitem[PS15]{Papadopoulos}
A.~Papadopoulos and W.~Su, \emph{On the {F}insler structure of
  {T}eichm\"{u}ller's metric and {T}hurston's metric}, Expo. Math. \textbf{33}
  (2015), no.~1, 30--47.

\bibitem[Ren76]{Renelt}
H.~Renelt, \emph{Konstruktion gewisser quadratischer differentiale mit hilfe
  von dirichletintegralen}, Math. Nachr. \textbf{73} (1976), no.~1, 125--1142.

\bibitem[Str66]{Strebel}
K.~Strebel, \emph{\"{U}ber quadratische {D}ifferentiale mit geschlossenen
  {T}rajektorien und extremale quasikonforme {A}bbildungen}, Festband 70.
  {G}eburtstag {R}. {N}evanlinna, Springer, Berlin, 1966, pp.~105--127.

\bibitem[Thu88]{Thurston}
W.P. Thurston, \emph{On the geometry and dynamics of diffeomorphisms of
  surfaces}, Bull. Amer. Math. Soc. (N.S.) \textbf{19} (1988), no.~2, 417--431.

\bibitem[Wal14]{WalshThurs}
C.~Walsh, \emph{The horoboundary and isometry group of {T}hurston's {L}ipschitz
  metric}, Handbook of {T}eichm\"uller theory, IRMA Lect. Math. Theor. Phys.,
  19, vol.~4, Eur. Math. Soc., Z\"urich, 2014, pp.~327--353.

\bibitem[Wal19]{WalshTeich}
\bysame, \emph{The asymptotic geometry of the {T}eichm\" uller metric}, Geom.
  Dedicata \textbf{200} (2019), no.~1, 115--152.

\end{thebibliography}

\end{document}